\documentclass[11pt,reqno]{amsart}

\newtheorem{proposition}{Proposition}[section]

\newtheorem{theorem}[proposition]{Theorem}

\theoremstyle{definition}
\newtheorem{definition}[proposition]{Definition}

\newtheorem{examples}[proposition]{Examples}
\newtheorem{remark}[proposition]{Remark}

\newcommand{\thlabel}[1]{\label{th:#1}}
\newcommand{\thref}[1]{Theorem~\ref{th:#1}}
\newcommand{\selabel}[1]{\label{se:#1}}
\newcommand{\seref}[1]{Section~\ref{se:#1}}

\newcommand{\exlabel}[1]{\label{ex:#1}}
\newcommand{\exref}[1]{Example~\ref{ex:#1}}
\newcommand{\delabel}[1]{\label{de:#1}}
\newcommand{\deref}[1]{Definition~\ref{de:#1}}
\newcommand{\eqlabel}[1]{\label{eq:#1}}
\newcommand{\equref}[1]{(\ref{eq:#1})}

\def\NN{{\mathbb N}}
\def\CC{{\mathbb C}}

\newcommand{\Cc}{\mathcal{C}}

\newcommand{\Mm}{\mathcal{M}}

\def\*C{{}^*\hspace*{-1pt}{\Cc}}
\def\text#1{{\rm {\rm #1}}}

\input xy
\xyoption {all} \CompileMatrices

\usepackage{amssymb}
\usepackage{color,amssymb,graphicx,amscd,amsmath}
\usepackage[colorlinks,urlcolor=blue,linkcolor=blue,citecolor=blue]{hyperref}

\begin{document}
\title[Classifying complements for associative algebras]
{Classifying complements for associative algebras}

\author{A. L. Agore}
\address{Faculty of Engineering, Vrije Universiteit Brussel, Pleinlaan 2, B-1050 Brussels, Belgium
\textbf{and} Department of Applied Mathematics, Bucharest
University of Economic Studies, Piata Romana 6, RO-010374
Bucharest 1, Romania} \email{ana.agore@vub.ac.be and
ana.agore@gmail.com}

\thanks{The author is supported by an \emph{Aspirant} Fellowship from the Fund for Scientific
Research-Flanders (Belgium) (F.W.O. Vlaanderen). This research is
part of the grant no. 88/05.10.2011 of the Romanian National
Authority for Scientific Research, CNCS-UEFISCDI}



\maketitle

\begin{abstract}
For a given extension $A \subset E$ of associative algebras we
describe and classify up to an isomorphism all $A$-complements of
$E$, i.e. all subalgebras $X$ of $E$ such that $E = A + X$ and $A
\cap X = \{0\}$. Let $X$ be a given complement and $(A, \, X, \,
\triangleright, \triangleleft, \leftharpoonup, \rightharpoonup
\bigl)$ the canonical matched pair associated with the
factorization $E = A + X$. We introduce a new type of deformation
of the algebra $X$ by means of the given matched pair and prove
that all $A$-complements of $E$ are isomorphic to such a
deformation of $X$. Several explicit examples involving the matrix
algebra are provided.
\end{abstract}

\section*{Introduction}

The concept of a \emph{matched pair} first appeared in the group
theory setting (\cite{Takeuchi}). Since then, the corresponding
concepts were introduced for several other categories such as Lie
algebras (\cite{Majid}), Hopf algebras (\cite{Majid2}), groupoids
(\cite{Andr}), Leibniz algebras (\cite{am-2013b}), locally compact
quantum groups (\cite{VV}), etc. With any matched pair of groups
(resp. Lie algebras, Hopf algebras, etc.) we can associate a new
group (resp. Lie algebra, Hopf algebra, etc.) called the
\emph{bicrossed product}. The bicrossed product construction is
responsible for the so-called \emph{factorization problem}, which
asks for the description and classification of all objects $E$
(groups, Lie algebras, Hopf algebras, etc.) which can be written
as a 'product' of two subobjects $A$ and $X$ having 'minimal
intersection' in $E$ - we refer to \cite{abm1} for more details, a
historical background and additional references. In the setting of
associative algebras, the bicrossed product was recently
introduced in \cite{am-2013c} as a special case of the more
general unified product. However, in this paper we use a slightly
more general construction than the one from \cite{am-2013c},
leaving aside the unitary condition on the algebras.

The classifying complements problem (CCP) was introduced in
\cite{am-2013a} in a very general, categorical setting, as a sort
of converse of the factorization problem. A similar problem,
called invariance under twisting, was studied in \cite{panaite2}
for Brzezinski's crossed products. In this paper we deal with the
(CCP) in the context of associative algebras:

\textbf{Classifying complements problem (CCP):} \textit{Let $A
\subset E$ be a given subalgebra of $E$. If an $A$-complement of
$E$ exists, describe explicitly, classify all $A$-complements of
$E$ and compute the cardinal of the (possibly empty) isomorphism
classes of all A-complements of E (which will be called the
factorization index $[E: A]^f$ of $A$ in $E$).}

Another related problem which will not be discussed in this paper
is that concerning the existence of complements whose natural
approach is the computational one. In the sequel the existence of
a complement will, however, be a priori assumed and we will be
interested in describing all complements of an algebra extension
$A \subset E$ in terms of one given complement.

The paper is organized as follows. In \seref{prel} we recall the
bicrossed product for associative algebras introduced in
\cite{am-2013c}. However, the construction used in this paper is
slightly more general as we drop the unitary assumption on the
algebras. \seref{2} contains the main results of the paper which
provide the complete answer to the (CCP) for associative algebras.
Let $A \subset E$ be a given extension of algebras. If $X$ is a
given $A$-complement of $E$ then \thref{descriereCompAlg} provides
the description of all complements of $A$ in $E$: any
$A$-complement of $E$ is isomorphic to an $r$-deformation of $X$,
as defined by \equref{rAlgdef}. In other words, exactly as in the
case of Hopf algebras, Lie algebras or Leibniz algebras, given $X$
an $A$-complement of $E$ all the other $A$-complements of $E$ are
deformations of the algebra $X$ by certain maps $r: X \to A$
associated with the canonical matched pair which arises from the
factorization $E = A + X$. The theoretical answer to the (CCP) is
given in \thref{clasformelorAlg} where we explicitly construct a
cohomological type object ${\mathcal H}{\mathcal A}^{2} (X, A \, |
\, (\triangleright, \triangleleft, \leftharpoonup,
\rightharpoonup) )$ which parameterizes all $A$-complements of
$E$. We introduce the factorization index $[E: A]^f$ of a given
extension $A \subset E$ as the cardinal of the (possibly empty)
isomorphism classes of all A-complements. Moreover, we prove that
the factorization index is computed by the formula: $[E: A]^f =
|\, {\mathcal H}{\mathcal A}^{2} (X, A \, | \, (\triangleright,
\triangleleft, \leftharpoonup, \rightharpoonup) ) \, |$. Several
explicit examples are provided. More precisely, we indicate
associative algebra extensions whose factorization index is $1$,
$2$ or $3$. We end the paper with an extension of index at least
$4$.

\section{Preliminaries}\selabel{prel}
Unless otherwise stated, all vector spaces, linear or bilinear
maps are over an arbitrary field $K$ of characteristic zero. A map
$f: V \to W$ between two vector spaces is called the trivial map
if $f (v) = 0$, for all $v\in V$. By an algebra $A$ we mean an
associative, not necessarily unital algebra over $K$. The concept
of left/right $A$-module or $A$-bimodule is defined as in the case
of unital algebras except of course for the unitary condition.
${}_A\Mm_A$ stands for the category of all $A$-bimodules, i.e.
triples $(V, \, \curvearrowright, \, \curvearrowleft)$ consisting
of a vector space $V$ and two bilinear maps $\curvearrowright \, :
A \times V \to V$, $\curvearrowleft: V \times A \to V$ such that
$(V, \curvearrowright)$ is a left $A$-module, $(V,
\curvearrowleft)$ is a right $A$-module and $a \curvearrowright (x
\curvearrowleft b) = (a \curvearrowright x) \curvearrowleft b$,
for all $a$, $b\in A$ and $x\in V$. \\ Let $A \subseteq E$ be a
subalgebra. Another subalgebra $X$ of $E$ is called an
\emph{$A$-complement of $E$} if $E = A + X$ and $A \cap X =
\{0\}$. For an arbitrary integer $n \geq 2$ let
$\mathcal{M}_{n}(K)$ be the algebra of $n \times n$ matrices over
the field $K$. We denote by $e_{i\,j} \in \mathcal{M}_{n}(K)$ the
matrix having $1$ in the $(i,j)^{th}$ position and zeros
elsewhere.

\subsection*{Bicrossed products revisited}
We recall the construction of the bicrossed product for
associative algebras as defined in \cite{am-2013c} but rephrased
into the present setting. More precisely, working with associative
not necessarily unital algebras will result in dropping the
normalization assumption on the matched pair.

\begin{definition} \delabel{mpalgebras}
A \emph{matched pair} of algebras is a system $(A, \, X,
\triangleright, \triangleleft, \leftharpoonup, \rightharpoonup
\bigl)$ consisting of two algebras $A$, $X$ and four bilinear maps
$$
\triangleleft : X \times A \to X, \quad \triangleright : X \times
A \to A, \quad \leftharpoonup \, : A \times X \to A, \quad
\rightharpoonup \, : A \times X \to X
$$
such that $(X, \rightharpoonup, \triangleleft) \in {}_A\Mm_A$ is
an $A$-bimodule, $(A, \triangleright, \leftharpoonup) \in
{}_X\Mm_X$ is an $X$-bimodule and the following compatibilities
hold for any $a$, $b \in A$, $x$, $y \in X$:
\begin{enumerate}
\item[(MP1)] $a \rightharpoonup (x\, y) = (a \rightharpoonup x) \,
y + (a \leftharpoonup x) \rightharpoonup y$;

\item[(MP2)] $ (a\,b) \leftharpoonup x = a\, (b \leftharpoonup x)
+ a \leftharpoonup (b \rightharpoonup x)$;

\item[(MP3)] $ x \triangleright (a\,b) = (x \triangleright a)\, b
+ (x \triangleleft a) \triangleright b$;

\item[(MP4)] $(x\, y) \triangleleft a = x \triangleleft (y
\triangleright  a) + x \,(y \triangleleft a)$;

\item[(MP5)] $a \,(x \triangleright b) + a \leftharpoonup (x
\triangleleft b) = (a \leftharpoonup x) \, b + (a\rightharpoonup
x) \triangleright b$;

\item[(MP6)] $x \triangleleft (a \leftharpoonup y) + x \, (a
\rightharpoonup y) = (x \triangleright a) \rightharpoonup y + (x
\triangleleft a) \, y $;
\end{enumerate}
\end{definition}

Let $(A, \, X, \triangleright, \triangleleft, \leftharpoonup,
\rightharpoonup \bigl)$ be a matched pair of algebras. Then, $A
\bowtie X = A \times X$, as a vector space, with the bilinear map
defined by
\begin{equation}\eqlabel{defbicr}
(a, \, x) \bullet (b, \, y) := \bigl( ab + a \leftharpoonup y + x
\triangleright b, \,\, a \rightharpoonup y  + x\triangleleft b +
x\, y \bigl)
\end{equation}
for all $a$, $b \in A$ and $x$, $y \in X$ is an associative
algebra called the \emph{bicrossed product} associated with the
matched pair $(A, \, X, \triangleright, \triangleleft,
\leftharpoonup, \rightharpoonup \bigl)$. As we will see in the
following examples, matched pairs of algebras appear quite
naturally from minimal sets of data.

\begin{examples}

$1)$ Let $A$ be an algebra and $(X, \rightharpoonup,
\triangleleft) \in {}_A\Mm_A$ an $A$-bimodule. We see $X$ as an
algebra with the trivial multiplication, i.e. $xy = 0$ for all
$x$, $y \in X$. It is straightforward to see that $(A, \, X, \,
\triangleleft, \, \triangleright_{0}, \, \leftharpoonup_{0}, \,
\rightharpoonup \bigl)$ is a matched pair of algebras, where
$\triangleright_{0}$ and $\leftharpoonup_{0}$ are the trivial
actions. The multiplication on the corresponding bicrossed product
$A \bowtie X$ is given as follows:
\begin{equation}
(a, \, x) \bullet (b, \, y) := \bigl( ab , \,\, a \rightharpoonup
y  + x\triangleleft b \bigl)
\end{equation}
The above bicrossed product is precisely the \emph{trivial
extension} of $A$ by the $A$-bimodule $X$.

$2)$ The previous example can be slightly generalized by
considering $A$ and $X$ to be both algebras such that $(X,
\rightharpoonup, \triangleleft) \in {}_A\Mm_A$ is an $A$-bimodule
for which the following compatibilities hold
\begin{equation}\eqlabel{semidirect}
a \rightharpoonup (x \, y) = (a \rightharpoonup x) \, y, \quad (x
\, y) \triangleleft a = x \, (y \triangleleft a), \quad x \, (a
\rightharpoonup y) = (x \triangleleft a) \, y
\end{equation}
for all $a\in A$, $x$, $y\in X$. In \cite[Definition, pg.
212]{pierce} a bimodule $X$ satisfying \equref{semidirect} is
called a \emph{multiplicative $A$-bimodule}. Then, the bicrossed
product associated with the matched pair $(A, \, X, \,
\triangleleft, \, \triangleright_{0}, \, \leftharpoonup_{0}, \,
\rightharpoonup \bigl)$, where $\triangleright_{0}$,
$\leftharpoonup_{0}$ are the trivial actions will be called,
following \cite[pg. 20]{am-2013c}, a \emph{semidirect product of
$A$ and $X$}. The multiplication on the corresponding bicrossed
product $A \bowtie X$ is given as follows:
\begin{eqnarray*}
(a, \, x) \bullet (b, \, y) := \bigl( ab , \,\, a \rightharpoonup
y  + x\triangleleft b + xy \bigl)
\end{eqnarray*}
This construction originates in \cite[Lemma a]{pierce} where is
presented in a different form.
\end{examples}


The bicrossed product of two algebras is the construction
responsible for the so-called \emph{factorization problem}, which
in the case of associative algebras comes down to:

\emph{Let $A$ and $X$ be two algebras. Describe and classify all
algebras $E$ that factorize through $A$ and $X$, i.e. $E$ contains
$A$ and $X$ as subalgebras such that $E = A + X$ and $A \cap X =
\{0\}$.}

Recall from \cite[Corollary 3.7]{am-2013c} that an algebra $E$
factorizes through two subalgebras $A$ and $X$ if and only if
there exists a matched pair of algebras $(A, \, X, \triangleright,
\triangleleft, \leftharpoonup, \rightharpoonup \bigl)$ such that $
E \cong A \bowtie X$. More precisely, if $E$ factorizes through
$A$ and $X$ we can construct a matched pair of algebras as
follows:
\begin{equation}\eqlabel{mpcanonic}
x \triangleright a + x \triangleleft a : = x \, a, \qquad a
\leftharpoonup x + a \rightharpoonup x := a \, x
\end{equation}
for all $a \in A$, $x \in X$. Throughout, the above matched pair
will be called \emph{the canonical matched pair} associated with
the factorization of $E$ through $A$ and $X$. Besides the
factorizable algebras mentioned above, several other classes of
associative algebras were studied recently: for instance, in
\cite{KO}, all complex finite-dimensional algebras of level one
are described.

\begin{examples}\exlabel{matrix}
$1)$ Let $n \in \NN$, $n \geq 2$. It can be easily seen that
$\mathcal{M}_{n}(K)$ factorizes through the subalgebra of strictly
lower triangular matrices $A = \{(a_{i\, j})_{i, \, j =
\overline{1,n}} ~|~ a_{i\, j} = 0 \,\, {\rm for} \,\, i \leq j \}$
and the subalgebra of upper triangular matrices $X = \{(x_{i\,
j})_{i, \, j = \overline{1,n}} ~|~ x_{i\, j} = 0 \,\, {\rm for}
\,\, i > j \}$. We denote by $\mathcal{B}_{A}:= \{e_{i\, j} ~|~
i,\, j \in \overline{1, n}, \, i > j\}$ and $\mathcal{B}_{X}:=
\{e_{i\, j} ~|~ i,\, j \in \overline{1, n}, \, i \leq j\}$ the
$K$-basis of $A$, respectively $X$. Then, the canonical matched
pair associated with this factorization is given as follows:
\begin{eqnarray*}
e_{i \, j} \leftharpoonup e_{l \, k} = \left\{
\begin{array}{rl} e_{i\, k}, & \mbox{if}\,\,\, i > k \geq j = l\\
0, & \mbox{otherwise}
 \end{array}\right.,
\qquad e_{i \, j} \rightharpoonup e_{l \, k} = \left\{
\begin{array}{rl} e_{i\, k}, & \mbox{if}\,\,\, l = j < i \leq k \\
0, & \mbox{otherwise}
 \end{array}\right.
\end{eqnarray*}

\begin{eqnarray*}
e_{r \, s} \triangleright e_{p \, t} = \left\{
\begin{array}{rl} e_{r\, t}, & \mbox{if}\,\,\, t < r \leq s = p\\
0, & \mbox{otherwise}
 \end{array}\right.,
\qquad e_{r \, s} \triangleleft e_{p \, t} = \left\{
\begin{array}{rl} e_{r\, t}, & \mbox{if}\,\,\, r \leq t < s = p \\
0, & \mbox{otherwise}
 \end{array}\right.
\end{eqnarray*}

$2)$ Consider $n \in \NN$, $n \geq 2$. Then $\mathcal{M}_{n}(K)$
factorizes also through the subalgebras $A = \{(a_{i\, j})_{i, \,
j = \overline{1,n}} ~|~ a_{n\, u} = 0 \,\, {\rm for} \,\, {\rm
all} \,\, u = \overline{1,n} \}$ and $X = \{(x_{i\, j})_{i, \, j =
\overline{1,n}} ~|~ x_{k\, l} = 0 \,\, {\rm for} \,\, {\rm all}
\,\, k = \overline{1, n-1} \,\, {\rm and} \,\, l = \overline{1, n}
\}$. We denote by $\mathcal{B}_{A}:= \{e_{i\, j} ~|~ i =
\overline{1,\, n-1}, \,\, j = \overline{1, n} \}$ and
$\mathcal{B}_{X}:= \{e_{n\, j} ~|~ j = \overline{1,\, n}\}$ the
$K$-basis of $A$, respectively $X$. The canonical matched pair
associated with this factorization is given as follows:
\begin{eqnarray*}
e_{n \, u} \triangleleft  e_{v \, t} = \left\{
\begin{array}{rl} e_{n\, t}, & \mbox{if}\,\,\, u = v\\
0, & \mbox{otherwise}
 \end{array}\right.,
\qquad e_{v \, t} \leftharpoonup e_{n \, u} = \left\{
\begin{array}{rl} e_{v\, u}, & \mbox{if}\,\,\, t = n \\
0, & \mbox{otherwise}
 \end{array}\right.
\end{eqnarray*}
while the other two actions are both trivial.

$3)$ Let $R$, $S$ be $K$-algebras and $M \in
{}_{R}\mathcal{M}_{S}$. Then the algebra $E = \begin{pmatrix} R & M\\
0 & S \end{pmatrix}$ factorizes through the subalgebras $A :=
\begin{pmatrix} R & 0 \\ 0 & S \end{pmatrix}$ and $X = \begin{pmatrix} 0 & M \\ 0 & 0
\end{pmatrix}$. More precisely, the associated matched pair is
given as follows for all $r \in R$, $s \in S$, $m \in M$:
\begin{eqnarray*}
\begin{pmatrix} 0 & m \\ 0 & 0 \end{pmatrix} \triangleleft \begin{pmatrix} r & 0 \\ 0 & s
\end{pmatrix} = \begin{pmatrix} 0 & ms \\ 0 & 0 \end{pmatrix} , \qquad
\begin{pmatrix} r & 0 \\ 0 & s \end{pmatrix} \,\, \rightharpoonup
\,\, \begin{pmatrix} 0 & m \\ 0 & 0 \end{pmatrix} =
\begin{pmatrix} 0 & rm \\ 0 & 0 \end{pmatrix}
\end{eqnarray*}
while the other two actions are both trivial.
\end{examples}

\section{Classifying complements. Applications}\selabel{2}

In this section we prove the main result of this paper which
answers the (CCP) for algebras. First we need to introduce the
following concept:

\begin{definition}
Let $(A, \, X, \triangleright, \triangleleft, \leftharpoonup,
\rightharpoonup \bigl)$ be a matched pair of algebras. A linear
map $r: X \to A$ is called a \emph{deformation map} of the matched
pair $(A, \, X, \triangleright, \triangleleft, \leftharpoonup,
\rightharpoonup \bigl)$ if the following compatibility holds for
all $x$, $y \in X$:
\begin{equation}\eqlabel{factAlg}
r(x) \, r(y) - r(x \, y) = r \bigl(r(x) \rightharpoonup y + x \lhd
r(y)\bigl) - r(x) \leftharpoonup y - x \rhd r(y)
\end{equation}
\end{definition}

We denote by ${\mathcal D}{\mathcal M} \, (A, X \, | \,
(\triangleright, \triangleleft, \leftharpoonup, \,
\rightharpoonup) )$ the set of all deformation maps of the matched
pair $(A, X, \triangleright, \triangleleft, \leftharpoonup,
\rightharpoonup)$. The trivial map $r: X \to A$, $r(x) = 0$, for
all $x\in X$ is of course a deformation map. The right hand side
of \equref{factAlg} measures how far $r: X\to A$ is from being an
algebra map.

The next example shows that computing all deformation maps
associated with a given matched pair is a highly non-trivial
problem.

\begin{examples}
Consider $\mathcal{M}_{n}(K)$ with the factorization given in
\exref{matrix} $1)$. Then ${\mathcal D}{\mathcal M} \, (A, X \, |
\, (\triangleright, \triangleleft, \leftharpoonup, \,
\rightharpoonup) )$ is in bijection with the families of scalars
$\{\bigl(\alpha^{ab}_{cd}\bigl)_{a, b, c, d \in \overline{1,n}}
~|~ \alpha^{ab}_{cd}\in K, \, c \leq d, \, a > b \}$ subject to
the compatibility condition:
\begin{eqnarray*}
\sum_{q < t < k}\, \alpha^{kt}_{ij} \, \alpha^{tq}_{rs} =
\delta_{jr} \, \alpha^{kq}_{is} + \sum_{r < u \leq s} \,
\alpha^{ur}_{ij} \, \alpha^{kq}_{us} \, + \, \sum_{i \leq v < j}
\, \alpha^{jv}_{rs} \, \alpha^{kq}_{iv} \, - \alpha^{kr}_{ij} \,
\delta_{sq} \, - \, \alpha^{jq}_{rs} \, \delta_{ki}
\end{eqnarray*}
for all $k > q$, $i \leq j$ and $r \leq s$, where $\delta_{jr}$ is
the Kroneker symbol. The bijection is such that any deformation
map $r: X \to A$ is implemented from a family of scalars by the
following formula:
$$
r(e_{ij}) = \sum_{k > t} \, \alpha^{kt}_{ij}\, e_{kt}, \quad {\rm
for} \quad {\rm all} \quad i \leq j
$$
\end{examples}

The following theorem is the key result in solving the (CCP) for
associative algebras:

\begin{theorem}\thlabel{descriereCompAlg}
Let $A$ be a subalgebra of $E$ and $X$ a given $A$-complement of
$E$ with the associated canonical matched pair $(A, X,
\triangleright, \triangleleft, \leftharpoonup, \rightharpoonup)$.

$(1)$ Let $r: X \to A$ be a deformation map of the above matched
pair. Then $X_{r} := X$, as a vector space, with the new
multiplication defined for any $x$, $y \in X$ by:
\begin{equation}\eqlabel{rAlgdef}
x \cdot_{r} y := x\,y + r(x) \rightharpoonup y + x \lhd r(y)
\end{equation}
is an associative algebra called the $r$-deformation of $X$.
Furthermore, $X_{r}$ is an $A$-complement of $E$.

$(2)$ $\overline{X}$ is an $A$-complement of $E$ if and only if
there exists an isomorphism of algebras $\overline{X} \cong
X_{r}$, for some deformation map $r: X \to A$ of the matched pair
$(A, X, \triangleright, \triangleleft, \leftharpoonup,
\rightharpoonup)$.
\end{theorem}

\begin{proof}
$1)$ The fact that the multiplication $\cdot_r$ defined by
\equref{rAlgdef} is associative follows by a long but
straightforward computation which relies on the axioms
(MP1)-(MP6). However, we present here a different and more natural
proof which will shed some light on the way we arrived at the
multiplication given by \equref{rAlgdef}. For a deformation map
$r: X \to A$, consider $f_{r}: X \to E = A \bowtie X$ to be the
$K$-linear map defined for any $x \in X$ by:
$$f_{r}(x) = (r(x),\, x)$$
We will prove that $\widetilde{X} : = {\rm Im}(f_{r})$ is an
$A$-complement of $E$. We start by showing that $\widetilde{X}$ is
a subalgebra of $A \bowtie X$. Indeed, for all $x$, $y \in X$ we
have:
\begin{eqnarray*}
\bigl(r(x), \, x\bigl) \bigl(r(y), \, y\bigl)
&\stackrel{\equref{defbicr}}{=}& \bigl(\underline{r(x)\, r(y) +
r(x)\leftharpoonup y + x \rhd r(y)}, \, r(x) \rightharpoonup y + x
\lhd r(y) + x \, y \bigl)\\
&\stackrel{\equref{factAlg}}{=}& \Bigl(r\bigl( r(x)
\rightharpoonup y + x \lhd r(y) + x \, y\bigl),\, r(x)
\rightharpoonup y + x \lhd r(y) + x \, y \Bigl)
\end{eqnarray*}
Therefore $\widetilde{X}$ is a subalgebra of $A \bowtie X$. We are
left to prove that $A \cap \widetilde{X} = \{0\}$. To this end,
consider $(a,\, x) \in A \cap \widetilde{X}$. Since in particular
we have $(a,\, x) \in \widetilde{X}$, it follows that $a = r(x)$.
As we also have $\bigl(r(x),\,x \bigl) \in A$ we obtain $x = 0$
and thus $A \cap \widetilde{X} = \{0\}$. Moreover, if $(b, \,y)
\in E = A \bowtie X$ we can write $(b, \,y) = \bigl(b - r(y), \,
0\bigl) + \bigl(r(y), \, y\bigl) \in A + \widetilde{X}$. Hence,
$\widetilde{X}$ is an $A$-complement of $E$. The proof will be
finished once we prove that $X_{r}$ and $\widetilde{X}$ are
isomorphic as algebras. Denote by $\widetilde{f}_{r}$ the linear
isomorphism from $X$ to $\widetilde{X}$ induced by $f_{r}$. We
will prove that $\widetilde{f}_{r}$ is an algebra morphism if we
consider $X$ endowed with the multiplication given by
\equref{rAlgdef}. For all $x$, $y \in X$ we have:
\begin{eqnarray*}
\widetilde{f}_{r}(x \cdot_{r} y)
&\stackrel{\equref{rAlgdef}}{=}&\widetilde{f}_{r} (x\,y + r(x)
\rightharpoonup y + x \lhd r(y))\\
&{=}& \Bigl(\underline{r\bigl( x \, y + r(x) \rightharpoonup y + x
\lhd r(y)\bigl)}, \,  x \, y + r(x) \rightharpoonup y + x \lhd
r(y)
\Bigl)\\
&\stackrel{\equref{factAlg}}{=}& \bigl(r(x)\, r(y) +
r(x)\leftharpoonup y + x \rhd r(y), \, x \, y + r(x)
\rightharpoonup y + x \lhd r(y)\bigl)\\
&{=}& \bigl(r(x), \, x\bigl) \bigl(r(y), \, y\bigl) =
\widetilde{f}_{r}(x) \widetilde{f}_{r}(y)
\end{eqnarray*}
Therefore, $X_{r}$ is an algebra and the proof is now finished.

$2)$ Let $\overline{X}$ be an arbitrary $A$-complement of $E$.
Since $E = A \oplus X = A \oplus \overline{X}$ we can find four
$K$-linear maps:
$$
u: X \to A, \quad v: X \to \overline{X}, \quad t:\overline{X} \to
A, \quad w: \overline{X} \to X
$$
such that for all $x \in X$ and $y \in \overline{X}$ we have:
\begin{equation} \eqlabel{lie111}
x = u(x) \oplus v(x), \qquad y = t(y) \oplus w(y)
\end{equation}
By an easy computation it follows that $v: X \to \overline{X}$ is
a linear isomorphism of vector spaces. We denote by $\tilde{v}: X
\to A \bowtie X$ the composition $\tilde{v} := i \circ v$ where
$i: \overline{X} \to E = A \bowtie X$ is the canonical inclusion.
Therefore, we have $\tilde{v}(x) \stackrel{\equref{lie111}}{=}
\bigl(-u(x),\, x\bigl)$, for all $x \in X$. In what follows we
will prove that $r := - u$ is a deformation map and $\overline{X}
\cong X_{r}$. Indeed, $\overline{X} = {\rm Im} (v) = {\rm Im}
(\tilde{v})$ is a subalgebra of $E = A \bowtie X$ and we have:
\begin{eqnarray*}
\bigl(r(x), \, x\bigl) \bigl(r(y), \, y\bigl)
&\stackrel{\equref{defbicr}}{=}& \bigl(r(x)\, r(y) +
r(x)\leftharpoonup y + x \rhd r(y), \, r(x) \rightharpoonup y + x
\lhd r(y) + x \, y \bigl)\\ &{=}& \bigl(r(z),\, z\bigl)
\end{eqnarray*}
for some $z \in X$. Thus, we obtain:
\begin{equation}\eqlabel{lie113}
r(z) = r(x) \, r(y) + r(x) \leftharpoonup y + x \triangleright
r(y), \qquad z = r(x) \rightharpoonup y + x \triangleleft r(y) + x
\, y
\end{equation}

By applying $r$ to the second part of \equref{lie113} it follows
that $r$ is a deformation map of the matched pair $(A, X,
\triangleright, \triangleleft, \leftharpoonup, \rightharpoonup)$.
Furthermore, \equref{rAlgdef} and \equref{lie113} show that $v:
X_{r} \to \overline{X}$ is also an algebra map. The proof is now
finished.
\end{proof}

\begin{remark}
We should point out that in the context of associative algebras
there exists another type of deformation in the literature, not
related to the one we introduce in \thref{descriereCompAlg} (see
for instance \cite{Fial0}).
\end{remark}

As we will see in \exref{3.7}, different deformation maps can give
rise to isomorphic deformations. Therefore, in order to classify
all complements we introduce the following:

\begin{definition}\delabel{equivAlg}
Let $(A, X, \triangleright, \triangleleft, \leftharpoonup,
\rightharpoonup)$ be a matched pair of algebras. Two deformation
maps $r$, $R: X \to A$ are called \emph{equivalent} and we denote
this by $r \sim R$ if there exists $\sigma: X \to X$ a $K$-linear
automorphism of $X$ such that for any $x$, $y\in X$ we have:
\begin{equation*}\eqlabel{equivAlgmaps}
\sigma (x \, y) - \sigma(x) \, \sigma(y) = \sigma(x) \triangleleft
R \bigl(\sigma(y)\bigl) + R \bigl(\sigma(x)\bigl) \rightharpoonup
\sigma(y) - \sigma\bigl(x \triangleleft r(y)\bigl) - \sigma
\bigl(r(x) \rightharpoonup y\bigl)
\end{equation*}
\end{definition}

The classification of complements now follows:

\begin{theorem}\thlabel{clasformelorAlg}
Let $A$ be a subalgebra of $E$, $X$ an $A$-complement of $E$ and
$(A, X, \triangleright, \triangleleft, \leftharpoonup,
\rightharpoonup)$ the associated canonical matched pair. Then
$\sim$ is an equivalence relation on the set $ {\mathcal
D}{\mathcal M} \, ( X, A \, | \, ( \triangleright, \triangleleft,
\leftharpoonup, \rightharpoonup ) )$ and the map
$$
{\mathcal H}{\mathcal A}^{2} (X, A \, | \, (\triangleright,
\triangleleft, \leftharpoonup, \, \rightharpoonup) ) \, := \,
{\mathcal D}{\mathcal M} \, (X, A \, | \, (\triangleright,
\triangleleft, \leftharpoonup, \rightharpoonup) )/ \sim \,
\longrightarrow {\mathcal F} (A, \, E), \qquad \overline{r}
\mapsto X_{r}
$$
is a bijection between ${\mathcal H}{\mathcal A}^{2} (X, A \, | \,
(\triangleright, \triangleleft, \leftharpoonup, \rightharpoonup)
)$ and the isomorphism classes of all $A$-complements of $E$. In
particular, the factorization index of $A$ in $E$ is computed by
the formula:
$$
[E : A]^f = | {\mathcal H}{\mathcal A}^{2} (X, A \, | \,
(\triangleright, \triangleleft, \leftharpoonup, \rightharpoonup)
)|
$$
\end{theorem}

\begin{proof}
Two deformation maps $r$ and $R$ are equivalent in the sense of
\deref{equivAlg} if and only if the corresponding algebras $X_r$
and $X_R$ are isomorphic. The conclusion follows by
\thref{descriereCompAlg}.
\end{proof}

We end the paper with a few examples illustrating our theory:

\begin{examples}\exlabel{3.7}
$1)$ Let $A$ be a two-sided ideal of $E$. Then $[E : A]^f \leq 1$.
Indeed, if an $A$-complement exists then it should be isomorphic
to the factor algebra $E/A$. Therefore, the factorization index is
at most $1$.

$2)$ Let $A = \begin{pmatrix} 0 & 0 \\
K & 0\end{pmatrix}$ be a subalgebra of $E = \mathcal{M}_{2}(K)$.
In this case $[E : A]^f= 2$. Indeed, using \exref{matrix} $1)$, it
follows that
$X = \begin{pmatrix} K & K \\
0 & K \end{pmatrix}$ is an $A$-complement of $E$. The non-zero
values of the canonical matched pair are given by:
\begin{eqnarray*}
e_{21} \, \leftharpoonup \, e_{11} = e_{21}, \quad e_{21} \,
\rightharpoonup \, e_{12} = e_{22}\\
e_{12} \, \triangleleft \, e_{21} = e_{11}, \quad e_{22} \,
\triangleright \, e_{21} = e_{21}
\end{eqnarray*}
By a straightforward computation it can be seen that the
associated deformation maps are as follows:
\begin{eqnarray*}
r_{a}(e_{11}) = a \, e_{21}, \qquad r_{a}(e_{12}) = a^{2} \,
e_{21}, \qquad r_{a}(e_{22}) = - a \, e_{21}, \qquad a \in K
\end{eqnarray*}
The multiplication on the $r_{a}$-deformation of $X$ is described
below:
\begin{center}
\begin{tabular} {c | c  c  c  }
$\cdot_{r_{a}}$ & $e_{11}$ & $e_{12}$ & $e_{22}$\\
\hline
$e_{11}$ & $e_{11}$ & $e_{12}+ a\,e_{22}$ & 0 \\
$e_{12}$ & $a \, e_{11}$ & $a^{2} \,e_{22}$ & $e_{12}$  \\
$e_{22}$ & 0 & $- a\, e_{22}$ & $e_{22}$ \\
\end{tabular}
\end{center}
If $a = 0$ then $r_{0}$ is the trivial map and $X_{r_{0}} = X$. On
the other hand, for any $a \in K^{*}$ we have an isomorphism of
algebras $X_{r_{a}}$ and $X_{r_{1}}$ given as follows:
\begin{eqnarray*}
\varphi: X_{r_{a}} \to X_{r_{1}}, \qquad \varphi(e_{11}) = e_{11},
\quad \varphi\left(a^{-1} \, e_{12} \right) = e_{12}, \quad
\varphi(e_{22}) = e_{22}
\end{eqnarray*}
To end with, it can be easily seen that $X_{r_{1}}$ is not
isomorphic to $X$. Indeed, it is enough to observe that $e_{11} +
e_{22}$ is a unit for $X$ while $X_{r_{1}}$ is not unital.
Therefore the factorization index is equal to $2$.

$3)$ Let $M \in\, _{K}\mathcal{M}_{K}$, and consider $A = \begin{pmatrix} K & 0 \\
0 & K\end{pmatrix}$ a subalgebra of $E = \begin{pmatrix} K & M \\
0 & K\end{pmatrix}$. According to \exref{matrix} $3)$ for $R = S
:= K$ we obtain that $X = \begin{pmatrix} 0 & M \\
0 & 0\end{pmatrix}$ is an $A$-complement of $E$. Then any
deformation map $r:
\begin{pmatrix} 0 & M \\ 0 & 0
\end{pmatrix} \to \begin{pmatrix} K & 0 \\ 0 & K \end{pmatrix} $ is uniquely implemented by two $K$-linear maps
$\alpha$, $\beta: M \to K$ such that $r
\begin{pmatrix} 0 & m \\ 0 & 0 \end{pmatrix} = \begin{pmatrix} \alpha(m) & 0 \\ 0 & \beta(m)
\end{pmatrix}$. It can be easily seen that $r$ satisfies
\equref{factAlg} if and only if $\alpha(m) \beta(n) = 0$ for all
$m$, $n \in M$. Therefore, we either have $\alpha(m) = 0$ for all
$m \in M$ or $\beta(m) = 0$ for all $m \in M$. If $\alpha(m) = 0$
for all $m \in M$ we obtain a deformation map $r_{\beta}$ defined
by $r_{\beta} \begin{pmatrix} 0 & m \\ 0 & 0 \end{pmatrix} =
\begin{pmatrix} 0 & 0 \\ 0 & \beta(m) \end{pmatrix} $ for all $m
\in M$, where $\beta : M \to K$ is an arbitrary $K$-linear map.
The multiplication induced by this deformation map is given as
follows:
$$
\begin{pmatrix} 0 & m \\ 0 & 0 \end{pmatrix} \cdot_{r_{\beta}}  \begin{pmatrix} 0 & n \\ 0 & 0
\end{pmatrix} = \begin{pmatrix} 0 & m \beta(n) \\ 0 & 0 \end{pmatrix}
$$
On the other hand, if $\beta(m) = 0$ for all $m \in M$ we obtain a
deformation map $r^{\alpha}$ defined by $r^{\alpha} \begin{pmatrix} 0 & m \\
0 & 0
\end{pmatrix} = \begin{pmatrix} \alpha(m) & 0 \\ 0 & 0
\end{pmatrix} $ for all $m \in M$, where $\alpha : M \to K$ is an
arbitrary $K$-linear map. The multiplication induced by this
deformation map is given as follows:
$$
\begin{pmatrix} 0 & m \\ 0 & 0 \end{pmatrix} \cdot_{r^{\alpha}}  \begin{pmatrix} 0 & n \\ 0 & 0
\end{pmatrix} = \begin{pmatrix} 0 & \alpha(m) n \\ 0 & 0 \end{pmatrix}
$$
By a straightforward computation it can be proved that if the
$K$-linear maps $\alpha$, $\alpha ' : M \to K$ are different from
the trivial map then $r^{\alpha}$ is equivalent in the sense of
\deref{equivAlg} to $r^{\alpha '}$. In the same manner, if the
$K$-linear maps $\beta$, $\beta ' : M \to K$ are different from
the trivial map then $r_{\beta}$ is equivalent in the sense of
\deref{equivAlg} to $r_{\beta '}$. Finally, $r^{\alpha}$ is never
equivalent to $r_{\beta}$ in the sense of \deref{equivAlg}, except
for the case when both $\alpha$ and $\beta$ are equal to the
trivial map. Therefore the factorization index $[E : A]^f= 3$.

$4)$ Let $A = \begin{pmatrix} K & K & K \\
K & K & K \\ 0 & 0 & 0 \end{pmatrix}$ be a subalgebra of $E =
\mathcal{M}_{3}(K)$. Then, by \exref{matrix} $2)$, it follows that
$X = \begin{pmatrix} 0 & 0 & 0 \\ 0 & 0 & 0 \\ K & K & K
\end{pmatrix}$ is an $A$-complement of $E$. The non-zero
values of the canonical matched pair are given as follows:
\begin{eqnarray*}
e_{31} \triangleleft e_{11} = e_{31}, \quad e_{31} \triangleleft
e_{12} = e_{32}, \quad e_{31} \triangleleft e_{13} = e_{33}\\
e_{32} \triangleleft e_{21} = e_{31}, \quad e_{32} \triangleleft
e_{22} = e_{32}, \quad e_{32} \triangleleft e_{23} = e_{33}\\
e_{13} \leftharpoonup e_{31} = e_{11}, \quad e_{13} \leftharpoonup
e_{32} = e_{12}, \quad e_{13} \leftharpoonup e_{33} = e_{13}\\
e_{23} \leftharpoonup e_{31} = e_{21}, \quad e_{23} \leftharpoonup
e_{32} = e_{22}, \quad e_{23} \leftharpoonup e_{33} = e_{23}
\end{eqnarray*}
In this case the computational complexity increases dramatically,
making it very difficult to compute all associated deformation
maps. However, we are still able to check by a straightforward
computation that the following maps are deformations of the above
canonical matched pair:
\begin{eqnarray*}
r_{1}: X \to A,&&  r_{1}(e_{33}) = e_{22}\\
r_{2}: X \to A,&&  r_{2}(e_{33}) = e_{11} + e_{22}\\
r_{3}: X \to A,&&  r_{3}(e_{31}) = e_{12}, \quad r_{3}(e_{33})=
e_{11} + e_{22}
\end{eqnarray*}
We denote by $X_{i}$ the $r_{i}$-deformation of $X$, for all $i =
\overline{1,3}$. The multiplication tables of the $X_{i}$'s, $i =
\overline{1,3}$, are depicted below:
\begin{eqnarray*}
X_{1}:&& e_{33} e_{31} =
e_{31},\,\,\, e_{33} e_{32} = e_{32} e_{33} = e_{32},\,\,\, e_{33} e_{33} = e_{33}\\
X_{2}:&& e_{31} e_{33} = e_{33} e_{31} =
e_{31}, \,\,\, e_{32} e_{33} = e_{33} e_{32} = e_{32},\,\,\, e_{33} e_{33} = e_{33}\\
X_{3}:&& e_{31} e_{31} = e_{32}, \,\,\, e_{31} e_{33} = e_{33}
e_{31} = e_{31}, \,\,\, e_{32} e_{33} = e_{33} e_{32} =
e_{32},\,\,\, e_{33} e_{33} = e_{33}
\end{eqnarray*}
It can be easily seen that the $X_{i}$'s, $i = \overline{1,3}$,
are isomorphic to the algebras A$s_{3}^{9}$, A$s_{3}^{10}$ and
respectively A$s_{3}^{12}$ listed in \cite{RRB} (for a complete
list of $3$-dimensional associative algebras over $\CC$ we refer
to \cite{Fial}). In each case the isomorphism sends $e_{31}$ to
$e_{1}$, $e_{32}$ to $e_{2}$ and $e_{33}$ to $e_{3}$, where
according to the notations of \cite{RRB}, $\{e_{1}, \, e_{2}, \,
e_{3}\}$ is a $K$-basis for the $3$-dimensional algebras mentioned
above. In particular, we obtain that the $X_{i}$'s are two by two
non-isomorphic. Moreover, none of the three algebras listed above
is isomorphic to $X$. Indeed, to start with, we should notice that
$X$ is not commutative and therefore it cannot be isomorphic to
the commutative algebras $X_{2}$ or $X_{3}$. We prove now that $X$
is not isomorphic to $X_{1}$. Assume that $\varphi: X_{1} \to X$
is an isomorphism of algebras given by:
\begin{eqnarray*}
\varphi(e_{31}) = \Sigma_{i=1}^{3} a_{i} e_{3i}, \quad
\varphi(e_{32}) = \Sigma_{i=1}^{3} b_{i} e_{3i},
\quad\varphi(e_{33}) = \Sigma_{i=1}^{3} c_{i} e_{3i}
\end{eqnarray*}
where $a_{i}$, $b_{i}$, $c_{i} \in K$ for all $i =
\overline{1,3}$. Since $\varphi(e_{31}) = \varphi(e_{33} e_{31}) =
\varphi(e_{33}) \varphi(e_{31}) = c_{3} \varphi(e_{31})$ and
$\varphi(e_{31}) \neq 0$ (as $\varphi$ is an isomorphism) we
obtain that $c_{3} = 1$. Moreover, from $0 = \varphi(e_{32}
e_{31}) = \varphi(e_{32}) \varphi(e_{31}) = b_{3} \varphi(e_{31})$
it follows that $b_{3} = 0$. Finally, as $\varphi(e_{32}) =
\varphi(e_{32} e_{33}) = \varphi(e_{32}) \varphi(e_{33}) = b_{3}
\varphi(e_{33}) = 0$ we have reached a contradiction. Therefore
$X_{1}$ is not isomorphic to $X$ and the factorization index $[E :
A]^f \geq 4$.
\end{examples}

\begin{remark}
After a careful analysis of \exref{3.7} we can easily conclude
that the deformations of a given algebra $X$ do not necessarily
preserve the properties of $X$. For instance, in \exref{3.7}, $2)$
we obtained a non-unital algebra as a deformation of a unital one,
while in \exref{3.7}, $4)$ we construct commutative deformations
of a non-commutative algebra. This is not the case for the
classical deformations studied in \cite{Fial0} as it is well known
that finite dimensional unital algebras only deform to unital
algebras.
\end{remark}

\end{document}